\documentclass[letterpaper,11pt]{amsart}

\usepackage{epsfig}
\usepackage{amssymb}
\usepackage{amsfonts}
\usepackage{amsmath}
\usepackage{graphicx}
\usepackage{mathrsfs}
\usepackage{stmaryrd}
\usepackage{amsthm}
\usepackage{mathtools}
\usepackage[all]{xy}
\usepackage[top=1in, bottom=1in, left=1in, right=1in]{geometry}

\newtheorem{theorem}{Theorem}
\newtheorem{lemma}{Lemma}
\newtheorem{proposition}{Proposition}

\newtheorem{conjecture}{Conjecture}

\title[On the positivity of scattering operators]{On the positivity of scattering operators for Poincar\'{e}-Einstein manifolds}

\author{Fang Wang}
\address{Shanghai Jiao Tong University, 800 Dongchuan Rd, Shanghai 200240, China.}
\email{fangwang1984@sjtu.edu.cn}
\date{August 30, 2016}
\thanks{The author was supported by Shanghai Pujiang Program No. 14PJ1405400}

\begin{document}

\begin{abstract}
In this paper, we mainly study the scattering operators for the Poincar\'{e}-Einstein manifolds. Those operators give the fractional GJMS operators $P_{2\gamma}$ for the conformal infinity. If a Poincar\'{e}-Einstein manifolds $(X^{n+1}, g_+)$ is locally conformally flat and there exists an representative $g$ for the conformal infinity $(M, [g])$ such that the scalar curvature $R$ is a positive constant and $Q_4>0$, then we prove that $P_{2\gamma}$ is positive for $\gamma\in (1,2)$ and thus the first real scattering pole is less than $\frac{n}{2}-2$. 
\end{abstract}
\maketitle

\section{Introduction}

Let $(X^{n+1}, g_+)$ be a Poincar\'{e}-Einstein manifold with smooth conformal infinity $(M, [g])$, i.e. 
$\overline{X}^{n+1}$ is a smooth manifold with boundary, $x$ is a smooth boundary defining function for  $\partial X=M$ and $g_+$ is a smooth Riemannian metric in the interior which satisfies 
$$
\begin{cases}
Ric_{g_+}=-ng_+ &\textrm{in}\  X, \\
x^2g_+|_{TM}\in [g] &\textrm{on}\ M. 
\end{cases}
$$
Here we require that $x^2g_+$ can be $C^{k,\alpha}$ extended to the boundary with $k\geq 2[\frac{n-1}{2}]+1$ and $0<\alpha<1$. 
Direct computation shows that all the sectional curvatures of $(X^{n+1}, g_+)$ converge to $-1$ when approaching to the boundary. 
A standard example is the Hyperbolic space $\mathbb{H}^{n+1}$ in the ball model:
$$
X^{n+1}=\{x\in\mathbb{R}^{n+1}: |z|<1\}, \quad g_+ =\frac{4dz^2}{(1-|z|^2)^2}=\frac{4(dr^2+r^2d\theta^2)}{(1-r^2)^2},
$$
where $(r,\theta)$ is the polar coordinates. Take the geodesic normal defining function 
$
x=\frac{2(1-r)}{1+r}. 
$
Then for $x\in (0,2)$
$$
g_+=x^{-2}\left(dx^2+\left(1-\frac{x^2}{4}\right)^2 d\theta^2 \right), \quad x^2g_+|_{T\mathbb{S}^n} = d\theta^2. 
$$

The spectrum and resolvent for the Laplacian-Beltrami operator of $(X^{n+1}, g_+)$ is studied by
Mazzeo-Melrose in \cite{MM1}, Mazzeo in \cite{Ma1} and Guillarmou \cite{Gu1}. Actually the authors dealt with more general asymptotic hyperbolic manifolds. 
They showed that 
$\mathrm{Spec}(\triangle_+)=\sigma_{pp}(\triangle_+) \cup \sigma_{ac}(\triangle_+) $,  where $\sigma_{pp}(\triangle_+) $ is the $L^2$-eigenvalue set and
$\sigma_{ac}(\triangle_+)$ is the absolute spectrum, and 
$$
\sigma_{pp}(\triangle_+) \subset \left(0,\frac{n^2}{4}\right), \quad \sigma_{ac}(\triangle_+)=\left[\frac{n^2}{4}, +\infty\right).
$$
For $s\in \mathbb{C}, \mathrm{Re}(s)>\frac{n}{2}, s(n-s)\notin \sigma_{pp}(\triangle_+)$, the resolvent $R(s)=(\triangle_+-s(n-s))^{-1}$ defines a bounded map
$$
R(s): L^2(dV_{g_+}) \longrightarrow  L^2(dV_{g_+}). 
$$
Moreover $R(s)$ can be meromorphically extended to $\mathbb{C}\backslash \{\frac{n-1}{2}-K-\mathbb{N}_0\}$. Here $K$ is an integer defined by that $2K$ is the even order of $g_+$ in its asymptotic expansion near the boundary. For Poincar\'{e}-Einstein metric $g_+$, $K\geq \frac{n-1}{2}$ for $n$ odd and $K\geq \frac{n-2}{2}$ for $n$ even, according to the regularity result given in \cite{CDLS}. In particular, for Hyperbolic space $\mathbb{H}^{n+1}$, $K=+\infty$. 
The $L^2$-eigenvalues can be estimated under certain geometric assumptions. For example, in \cite{Le1}, Lee showed if $(X^{n+1}, g_+)$ is Poincar\'{e}-Einstein and its conformal infinity is of nonnegative Yamabe type, then $\sigma_{pp}(\triangle_+)=\emptyset$. 

The scattering operators associate to $(X^{n+1}, g_+)$ are define in the following way. 
Consider
$$
(\triangle_+ - s(n-s))u=0, \quad x^{s-n}u|_{M}=f\in C^{\infty}(M). 
$$
If $\mathrm{Re}(s)>\frac{n}{2},  s(n-s)\notin \sigma_{pp}(\triangle_+),  2s-n\notin \mathbb{N}$, then
$$
u=x^{n-s} F + x^sG, \quad F, G\in C^{k, \alpha}(\overline{X}), \quad F|_{M}=f. 
$$
We define the scattering operator $S(s)$ by
$$
S(s): C^{\infty}(M)\longrightarrow C^{\infty}(M), \quad S(s)f=G|_{M}. 
$$
Here $S(s)$ is a one parameter family of conformally invariant elliptic pseudo-differential operators or order $2s-n$, which can be meromorphically extended to $\mathbb{C}\backslash \{\frac{n-1}{2}-K-\mathbb{N}_0\}$ with $K$ the same as before. If $s_0>\frac{n}{2}$ is a pole satisfying $ 2s_0-n\in\mathbb{N},s_0(n-s_0)\notin\sigma_{pp}(\triangle_+)$, then the order of this pole is at most $1$ and the residue is a differential operator on $M$. In particular, if $(X^{n+1}, g_+)$ is a Poincar\'{e}-Einstein manifold and $\frac{n^2}{4}-k^2\notin\sigma_{pp}(\triangle_+)$ for $k\leq\min\{\frac{n}{2}, K\}$, then
$$
\mathrm{Res}_{s=s_0} S(s)=c_kP_{2k}, \quad c_k=\frac{(-1)^{k-1}}{2^{2k}k!(k-1)!}. 
$$
Here $P_{2k}$ is the GJMS operator of order $2k$ on $(M,g)$ with  $g=x^2g_+|_{TM}$. In particular, 
$P_2$ is the conformal Laplacian and $P_4$ is the Paneitz operator on $(M, g)$. 
 See \cite{JS1}, \cite{GZ1} for more details. 
 
 For simplicity, we define the renormalised scattering operators by
$$
P_{2\gamma} =d_{\gamma}S\left(\frac{n}{2}+\gamma\right), \quad d_{\gamma}=2^{2\gamma}\frac{\Gamma(\gamma)}{\Gamma(-\gamma)}.
$$
While $\gamma$ is not an integer, $P_{2\gamma}$ is also called the \textit{fractional GJMS operators}. Similarly, the fractional Q-curvature is defined by
$$
Q_{2\gamma}=\frac{2}{n-2\gamma} P_{2\gamma}1. 
$$
From the definition of $P_{2\gamma}$,  when $\gamma$ is not an integer, $P_{2\gamma}$ should depend on the interior metric $(X^{n+1}, g_+)$, not only on $(M, [g])$. 
A special case is the Hyperbolic space $\mathbb{H}^{n+1}$, which has conformal infinity $(\mathbb{S}^n, [g_c])$ where  $g_c=d\theta^2$ is the canonical sphere metric. 
In this case, the rigidity theorems given in  \cite{ST1} \cite{DJ1} and \cite{LQS1} tell us if $(X^{n+1}, g_+)$ is Poincar\'{e}-Eisntein with conformal infinity $(\mathbb{S}^n, [g_c])$, then $(X^{n+1}, g_+)$ must be the Hyperbolic space $\mathbb{H}^{n+1}$. 
So the fractional GJMS operators are uniquely defined, which are given in the following:
$$
\begin{aligned}
&P^{g_c}_{2\gamma}=\frac{\Gamma(B+\frac{1}{2}+\gamma)}{\Gamma(B+\frac{1}{2}-\gamma)}, \quad\mathrm{where}\quad B=\sqrt{\triangle_{g_c}+\left(\frac{n-1}{2}\right)^2}.
\end{aligned}
$$
Similarly, the fractional Q-curvature can be computed explicitly: 
$$
Q^{g_c}_{2\gamma}=\frac{2}{n-2\gamma} \frac{\Gamma(\frac{n}{2}+\gamma)}{\Gamma(\frac{n}{2}-\gamma)}. 
$$
Using the fractional GJMS operators, we can define the fractional Yamabe Invariant by
$$
\begin{aligned}
Y_{\gamma}(M, [g])=&\ \inf_{f\in C^{\infty}(M)} \frac{\int_{M} fP_{2\gamma} f \mathrm{dvol}_{g}}{\big(\int_M |f|^{\frac{2n}{n-2\gamma}} \mathrm{dvol}_{g}\big)^{\frac{n-2\gamma}{n}}}
\\
=&\ \inf_{\hat{g}\in [g]} \frac{\frac{n-2\gamma}{2}\int_{M}Q^{\hat{g}}_{2\gamma} \mathrm{dvol}_{\hat{g}}}{\left(\int_M \mathrm{dvol}_{\hat{g}}\right)^{\frac{n-2\gamma}{2}}}. 
\end{aligned}
$$
While $\gamma=1$, this is the classical Yamabe invariant.

We are mainly interested in the positivity of these renormalised scattering operators.  For $\gamma\in(0,1)$ the positivity of $P_{2\gamma}$ was studied by Guillarmou-Qing in \cite{GQ1}. 
\begin{theorem}[Guillarmou-Qing]\label{thm.gq1}
Suppose $(X^{n+1},g_+)$  ($n\geq 3$) is a Poincar\'{e}-Einstein manifold with conformal infinity $(M, [g])$. Fix a representative $g$ for the conformal infinity and assume  the scalar curvature $R$ is positive  on $(M, g)$. Then for  $\gamma\in (0, 1)$,
\begin{itemize}
\item[(a)] $Q_{2\gamma}>0$ on $M$;
\item[(b)] The first eigenvalue of $P_{2\gamma}$ is positive;
\item[(c)] The Green function of $P_{2\gamma}$ is positive;
\item[(d)] The first eigenspace of $P_{2\gamma}$ is spanned by a positive function.
\end{itemize}
\end{theorem}

Based on the positivity results and the identity $S(n-s)S(s)=\mathrm{Id}$, the authors also showed that
\begin{theorem}[Guillarmou-Qing]\label{thm.gq2}
Suppose $(X^{n+1},g_+)$ ($n\geq 3$) is a Poincar\'{e}-Einstein manifold with conformal infinity $(M, [g])$. Then the Yamabe invariant $\mathcal{Y}_1(M, [g])$ is positive if and only if the first scattering pole is less than $\frac{n}{2}-1$
\end{theorem}

The first scattering pole has more interesting  interpretation in the case of $X=\Gamma\backslash \mathbb{H}^{n+1}$, while $\Gamma$ is a convex co-compact group without torsion of orientation preserving isometries of $\mathbb{H}^{n+1}$. In this case, the conformal infinity is locally conformally flat and given by the quotient $M=\Gamma\backslash \Omega(\Gamma)$ where $\Omega(\Gamma)\subset \mathbb{S}^n$ is the domain of discontinuity of $\Gamma$. In \cite{Pe1}, Perry proved that the latest real scattering pole is given by the Poincar\'{e} exponent; and Sullivan \cite{Su1} and Patterson \cite{Pa1} showed that  the Poincar\'{e} exponent of the group $\Gamma$ is equal to  the Hausdorff dimension 
$\delta_{\Gamma}$ of the limit set $\Lambda(\Gamma)=\mathbb{S}^n\backslash \Omega(\Gamma)$. Due to the work of Schoen-Yau \cite{SY1} and Nayatani \cite{Na1}, 
$\delta_{\Gamma}$ is less than $\frac{n}{2}-1$ if and only if the conformally infinity is of positive Yamabe Type.

In this paper, we mainly study the scattering operators of order between $2$ and $4$, and prove the following theorem.

\begin{theorem}\label{thm.1}
Suppose $(X^{n+1},g_+)$ ($n\geq 5$) is a locally conformally flat Poincar\'{e}-Einstein manifold with conformal infinity $(M, [g])$  and fix a representative $g$ for the conformal infinity. Assume  the scalar curvature $R$ is a positive constant and $Q_4\geq  0$ on $(M,g)$. Then for  $\gamma\in (1,2)$,
\begin{itemize}
\item[(a)] $Q_{2\gamma}>0$ on $M$;
\item[(b)] The first eigenvalue of $P_{2\gamma}$ is positive;
\item[(c)] $P_{2\gamma}$ satisfies the strong maximun principle, i.e. for $f\in C^{\infty}(M)$,  $P_{2\gamma}f\geq 0$ implies $f>0$ or $f\equiv 0$;
\item[(d)] The Green function of $P_{2\gamma}$ is positive;
\item[(e)] The first eigenspace of $P_{2\gamma}$ is spanned by a positive function.
\end{itemize}
In this case the first scattering pole $s_0\leq \frac{n}{2}-2$. Furthermore,  if  $Q_4(p)>0$ at some point $p\in M$, then   the first scattering pole $s_0< \frac{n}{2}-2$. 
\end{theorem}

For the special case $X=\Gamma\backslash \mathbb{H}^{n+1}$ with conformal infinity $M=\Gamma\backslash \Omega(\Gamma)$, Theorem \ref{thm.1} implies that if there exists a representative $g$ such that the scalar curvature $R$ is a positive constant and $Q_4\geq  0$, then the Hausdorff dimension $\delta_{\Gamma}$ of the limit set is less than $\frac{n}{2}-2$. A similar result is given by Zhang \cite{Zh1} that if $R>0$ and $Q_{2\gamma}>0$ for some $\gamma\in(1,2)$, then $\delta_{\Gamma}< \frac{n}{2}-\gamma$.  

At last, we want to point out that the condition "$R$ is a positive constant" is a technical requirement. We expect to replace it by $R>0$ in the future. Based on the work of Gursky-Lin \cite{GL1}, we also make an conjecture as follows:

\begin{conjecture}\label{conj.1}

Suppose $(X^{n+1},g_+)$ ($n\geq 5$) is a  Poincar\'{e}-Einstein manifold with conformal infinity $(M, [g])$ . If the conformal infinity satisfies $\mathcal{Y}_1(M,  [g])>0,\mathcal{Y}_2(M,  [g])>0$, then the first scattering pole is less than $n/2-2$. 

\end{conjecture}

\vspace{0.2in}

\section{Asymptotic computations}
Supoose  $(X^{n+1},g_+)$ is a Poincar\'e-Einstein manifold with conformal infinity $(M,[g])$ and $n\geq 5$.  In this section we do some asymptotic computations for the Poinar\'{e}-Einstein manifolds. Fix a representative $g$ for the conformal infinity and choose $x$ to be the geodesic normal defining function, i.e.  $|dx|^2_{x^2g_+}=1$ in a neighbourhood of the boundary and $x^2g|_{TM}=g$. Then due to \cite{CDLS}, $g_+$ has a Taylor expansion near the boundary, i.e.
$$
\begin{gathered}
g_+=\frac{1}{x^2}\left(dx^2+g_0+x^2g_2+x^4g_4+O(x^{5})\right),\quad \mathrm{where}\quad\\
g_0=g,\quad
[g_2]_{ij}=- A_{ij},\quad
[g_4]_{ij}=\frac{1}{4(n-4)}(-B_{ij}+(n-4)A_i^k A_{jk}).
\end{gathered}
$$
Here $A$ is the Schouten tensor on $(M, g)$: 
$$
\begin{gathered}
A_{ij}=\frac{1}{n-2}\left( R_{ij}-Jg_{ij}\right),\quad 
J=\frac{R}{2(n-1)}=Q_2;
\end{gathered}
$$
and 
$$
B_{ij}=C_{ijk,}^{\ \ \ \ k}-A^{kl}W_{kijl}; \quad
C_{ijk}=A_{ij,k}-A_{ik,j}. 
$$
In this paper, we use "$,$" to denote the covariant derivatives w.r.t. boundary metric $g$ and "$:$"   the covariant derivatives w.r.t. interior metric $g_+$. 
Then the Laplacian of $g_+$,  denoted by $\triangle_+$, also has an expansion near the boundary, i.e. 
$$
\triangle_{+}=-(x\partial_x)^2+nx\partial_x + x^2L_2+x^4L_4+O(x^5)
$$
where
$$
\begin{gathered}
L_2=\triangle_{g}+Jx\partial_x, \\
L_4=\delta_g Ad +\frac{1}{2}(-\delta_g Jd+J\triangle_{g})+\frac{1}{2}|A|_g^2x\partial_x
\end{gathered}
$$
We also denote
$$
\begin{gathered}
L_2(s)=\triangle_{g}+sJ, 
\\
L_4(s)=\delta_g Ad +\frac{1}{2}(-\delta_g Jd+J\triangle_{g})+\frac{1}{2}s|A|_g^2
\end{gathered}
$$
For $f\in C^{\infty}(M)$ and $\mathrm{Re}(s)>\frac{n}{2}$, $2s-n\notin \mathbb{N}$ and $s(n-s)\notin \sigma_{pp}({\triangle_+})$, consider the equation
$$
\begin{gathered}
\triangle_{g_+}u-s(n-s)u=0, \quad x^{s-n}u|_{M}=f.
\end{gathered}
$$
Then $u=x^{n-s} F+x^sG$ with $F, G\in C^{\infty}(\overline{X})$. Moreover, $F$ has asymptotical expansion
$$
\begin{gathered}
F=f-\frac{x^2}{2(2s-n-2)}T_2(n-s)f+\frac{x^4}{8(2s-n-2)(2s-n-4)}T_4(n-s)f+O(x^5),
\end{gathered}
$$
and $G=S(s)f+O(x^2)$. 
Here
$$
\begin{aligned}
T_2(n-s)=&\ L_2(n-s),\\
T_4(n-s)=&\ L_2(n-s+2)L_2(n-s)-2(2s-n-2)L_4(n-s). 
\end{aligned}
$$
By direct computation, 
$$
\begin{gathered}
T_2(n-s)=P_2+\frac{n+2-2s}{2}J, 
\\
T_4(n-s)= P_4+ \frac{n+4-2s}{2}\left(2J\triangle_{g}+4\delta_{g} Ad +Q_4+(n-s)(J^2+2|A|_{g}^2)\right). 
\end{gathered}
$$
If $s=\frac{n}{2}+1$, then
$$
T_2\left(\frac{n}{2}-1\right)=P_2=\triangle_g +\frac{n-2}{2}J. 
$$
If $s=\frac{n}{2}+2$, then
$$
\begin{gathered}
T_4\left(\frac{n}{2}-2\right)=P_4=\triangle_g^2+ \delta_g\left((n-2)J-4A\right)d+\frac{n-4}{2}Q_4. 
\end{gathered}
$$

\vspace{0.2in}
\section{The positivity of $P_{2\gamma}$ for $\gamma\in (1,2)$}
We mainly prove Theorem \ref{thm.1} in this section. The positivity of fractional GJMS operators is studied carefully in Section 7 of \cite{CC1}. Combine their results with the spectrum theorem in \cite{Le1}, we first know that

\begin{proposition}[Case-Chang, Lee]\label{prop.2.1}
Let $(X^{n+1}, g_+)$ ($n\geq 4$)  be a Poinare\'{e}-Einstein manifold with conformal infinity $(M, [g])$. Fix a representative $g$ for the conformal infinity. Assume the scalar curvature $R>0$ and $Q_{2\gamma}>0$ for some $\gamma\in (1,2)$. Then
\begin{itemize}
\item[(a)] There is no $L^2$-eigenvalue for $\triangle_+$, i.e.   $\mathrm{spec}(\triangle_+)=[n^2/4,\infty)$;
\item[(b)] The first eigenvalue of $P_{2\gamma}$ satisfies $\lambda_1(P_{2\gamma})\geq \min_{M}Q_{2\gamma}>0$;
\item[(c)] $P_{2\gamma}$ satisfies strong maximum principle: if $P_{2\gamma}f\geq 0$ for $f\in C^{\infty}(M)$, then $f>0$ or $f\equiv 0$;
\item[(d)] The Green's function of $P_{2\gamma}$ is positive. 
\end{itemize}
\end{proposition}

Therefore, to prove Theorem \ref{thm.1}, we only need to prove part $(a)$ and part $(e)$. Here we work out a comparison theorem similar as Guillarmou-Qing did in \cite{GQ1}. 

For $\gamma\in(1,2)$, set $s=\frac{n}{2}+\gamma$  and   let $u$ solves the following equation:
\begin{equation}\label{eq.1}
\triangle_+u-s(n-s)u=0, \quad x^{s-n}u|_{M}=1.
\end{equation}
Then $u$ is positive on $X$  and near the  boundary $u$ has an asymptotic expansion as follows:
$$
\begin{gathered}
u=x^{n-s}\left(1+x^2u_2+x^{2\gamma}u_{2\gamma}+x^4u_4+O(x^{5})\right), 
\end{gathered}
$$
where
$$
\begin{gathered}
u_{2\gamma}=S(s)1=2^{-2\gamma}\frac{\Gamma(-\gamma)}{\Gamma(\gamma)} \frac{n-2\gamma}{2}Q_{2\gamma}, \\
u_2=-\frac{(n-s)}{2(2s-n-2)}J< 0, \\
u_4=\frac{(n-s)}{8(2s-n-2)(2s-n-4)}\left(Q_4+\frac{n+4-2s}{2}\left(J^2+2|A|_g^2\right)\right)< 0. 
\end{gathered}
$$
We also define a test function $\psi$ by
\begin{equation}\label{eq.2}
\psi=v^{\frac{n-s}{n-\lambda}} 
\quad\mathrm{where}\quad
\begin{cases}
\triangle_+v-\lambda(n-\lambda)v=w, & x^{\lambda-n}v|_{M}=1;\\
\triangle_+w-(\lambda-2)(n-\lambda+2)w=0, & x^{\lambda+2-n}w|_{M}=w_0. 
\end{cases}
\end{equation}
Here $\lambda>\max\{s+2, n\}$ will be fixed later  and
$$
w_0=\frac{-2(\lambda-s)(n-\lambda)}{2s-n-2}J>0. 
$$
So by maximum principle, $w>0$,  $v>0$ and hence $\psi>0$  on $X$. 

\begin{lemma} Near the boundary $M$, $\psi$ has an asymptotic expansion 
$$
\psi=x^{n-s}\left(1+\psi_2x^2+\psi_4x^4+O(x^{5})\right), 
$$
sastisfying
$$
\psi_2=u_2. 
$$
Assume $J>0$ and $Q_4\geq 0$, then
$$
\psi_4>u_4. 
$$
\end{lemma}
\begin{proof}
Here $v, w$ have asymptotic expansions as follows
$$
\begin{aligned}
&v=x^{n-\lambda}\left(1+v_2x^2+v_4x^4+O(x^{5})\right),
\\ 
&w=x^{n-\lambda+2}\left(w_0+w_2x^2+O(x^{3})\right). 
\end{aligned}
$$
By direct computation
$$
\begin{gathered}
w_2=\frac{(\lambda-s)(n-\lambda)}{(2\lambda-n-6)(2s-n-2)}\left(\triangle_gJ+(n-\lambda+2)J^2\right),\\
v_2=-\frac{(n-\lambda)}{2(2s-n-2)}J,\\
v_4=\frac{1}{4(2\lambda-n-4)}\left[\left(\frac{\lambda-s}{2\lambda-n-6}+\frac{1}{2}\right)\frac{n-\lambda}{2s-n-2}\left(\triangle_g J+(n-\lambda+2)J^2\right)-\frac{n-\lambda}{2}|A|^2\right].
\end{gathered}
$$
This implies 
$$
\begin{gathered}
\psi_2=-\frac{(n-s)}{2(2s-n-2)}J,\\
\psi_4=\frac{n-s}{n-\lambda}v_4+\frac{(n-s)(\lambda-s)}{2(n-\lambda)^2}v_2^2.
\end{gathered}
$$
So $\psi_2=u_2$ and 
$$
\begin{gathered}
\begin{aligned}
\psi_4-u_4=&\ \frac{(n-s)(\lambda-s)}{4}\left[\frac{Q_4}{(2\lambda-n-4)(2s-n-2)}\left(\frac{1}{2\lambda-n-6}-\frac{1}{2s-n-4}\right)\right]\\
&+ \frac{(n-s)(\lambda-s)}{4(2s-n-2)^2(2\lambda-n-6)}\left[(2s-n-2)|A|_g^2+(\lambda-s-2)J^2\right]. 
\end{aligned}
\end{gathered}
$$
Since $\lambda>s+2$,  $\lambda>n>s$ and $J>0$, $Q_4\geq 0$, we have $\psi_4-u_4>0$.   
\end{proof}

For the global behaviour of $\psi$,  a direct computation gives
$$
\begin{aligned}
\frac{\triangle_+\psi}{\psi}-s(n-s)
&=(n-s)\left[(\lambda-s)\left(1-\frac{1}{(n-\lambda)^2}v^{-2}|\nabla v|_+^{2}\right)+\frac{1}{n-\lambda}v^{-1}w\right]
\\
&=(n-s)Kx^4+O(x^{5}),
\end{aligned}
$$
where
$$
\begin{aligned}
K=&\ -(\lambda-s)\left[
\frac{1}{(2s-n-2)^2}J^2-\frac{1}{2\lambda-n-4}|A|_g^2\right.
\\
&\ \left. +
\frac{2(\lambda-s-1)}{(2\lambda-n-4)(2\lambda-n-6)(2s-n-2)}\big(\triangle_g J+(n-\lambda+2)J^2\big)\right]
\\
=&\  -(\lambda-s)\left[ \frac{(n+4-2s)}{(2\lambda-n-6)(2s-n-2)} \left((\lambda-s-2)J^2+|A|_g^2\right)\right.
\\
&\ + \left. \frac{2(\lambda-s-1)}{(2\lambda-n-4)(2\lambda-n-6)(2s-n-2)}Q_4
\right].
\end{aligned}
$$

\begin{lemma}\label{lem.2}
Assume  $\lambda>s+2$, $\lambda>n>s$ and $J>0$, $Q_4\geq 0$ on $M$, then $K<0$. 
\end{lemma}

Next take $\lambda=n+2$ in (\ref{eq.2}) to simplify the computations. In this case, $v^{-2}=O(x^4)$. Denote
$$
\begin{gathered}
\frac{\triangle_+\psi}{\psi}-s(n-s)=(n-s)v^{-2}I, 
\end{gathered}
$$
where
\begin{equation}\label{eq.3}
I=(n+2-s)\left(v^2-\frac{1}{4}|\nabla v|_+^{2}\right)-\frac{1}{2}(vw).
\end{equation}
Obviously $I|_{M}=K$.

\begin{lemma} \label{lem.3}
For $I$ defined in (\ref{eq.3})
we have
$$
\begin{aligned}
\triangle_+I=&(n+2-s)\left(\frac{1}{2}|\nabla^2v|^2_+ +\frac{n}{2}|\nabla v|_+^2-4(n+2)v^2\right)
\\ &
  -\frac{(n-s)}{2}\langle \nabla v, \nabla w\rangle_+   +[3(n+2)-2s]vw-\frac{1}{2}w^2. 
\end{aligned}
$$
\end{lemma}
\begin{proof}
Here on $X^{n+1}$, $R^+_{ij}=-n[g_+]_{ij}$ and $v, w$ satisfy
$$
\begin{aligned}
& v_{:k}^{\ \ k}=2(n+2)v-w, 
\\
& w_{:k}^{\ \ k}=0.
\end{aligned}
$$
Here the $:$ denotes the covariant derivative w.r.t. $g_+$.
Hence 
$$
\begin{aligned}
& v_{:ik}^{\ \ \ k}=v^{\ k}_{:\ ki}+R^+_{ik}v^{k}=(n+4)v_i -w_i, 
\\
& w_{:ik}^{\ \ \ k}=w^{\ k}_{:\ ki}+R^+_{ik}w^k=-nw_i. 
\end{aligned}
$$
Direct computation shows that
$$
(v^2)_{:k}^{\ \ k}=2\big(v^kv_k+vv_{:k}^{\ \ k}\big)=2\big(|\nabla v|^2_+  +v[2(n+2)v-w]\big),
$$
$$
(|\nabla v|^2_+)_{:k}^{\ \ k} =2\big(v_{:}^{\ ik}v_{:ik}+ v^iv_{:ik}^{\ \ \ k}\big)=2\big(|\nabla^2v|^2_+ +(n+4)|\nabla v|^2_+-\langle \nabla v, \nabla w\rangle_+\big), 
$$
$$
(vw)_{:k}^{\ \ k} =v_{:k}^{\ \ k}w+2v_kw^k+vw_{:k}^{\ \ k}= 2\langle \nabla v, \nabla w\rangle_+ +[2(n+2)v-w]w. 
$$
Hence we get  the formula for $\triangle_+I=-I_{:k}^{\ \ k}$. 
\end{proof}

Denote by
\begin{equation}\label{eq.4}
II=\triangle_+I. 
\end{equation}
Using the asymptotical expansion of $\triangle_+$, it is obvious that
$
II|_{M}=0. 
$

\begin{lemma}\label{lem.4}
For $II$ defined in (\ref{eq.4}), 
$$
\begin{aligned}
\triangle_+ II+2(n+2)II =& (n+2-s)\big(-|\nabla^3v|^2_++(6n+16)|\nabla v|^2_++2W^+_{ikjl}v^{ij}v^{kl} \big) 
\\
&+2(n+1-s) \langle \nabla^2v, \nabla^2w\rangle_+ -4(n+3-s)\langle \nabla v, \nabla w\rangle_++\frac{n-s-2}{2}|\nabla w|^2_+. 
\end{aligned}
$$
\end{lemma}
\begin{proof}
Let $R^+_{imjk}$ be the Riemannian curvature tensor of $(X^{n+1}, g_+)$. Then since $R^+_{ij}=-n[g_+]_{ij}$,  
$$R^+_{imjk}=W^+_{imjk}-[g_+]_{ij}[g_+]_{mk}+[g_+]_{ik}[g_+]_{mj},$$
Moreover, $R^+_{ij:k}=0$ and $R_{imjk:}^{+\ \ \ \ \ k}=0$. The later is because of Bianchi identity
$$
R^+_{jkim:l}+R^+_{jkml:i}+R^+_{jkli:m}=0. 
$$
Therefore, 
$$
\begin{aligned}
v_{:}^{\ ij}v_{:ijkl}=& v{:}^{\ ij}(v_{:ikj}-R^+_{imjk}v^m)_{:l}
\\
=&v_{:}^{\ ij}(v_{:kijl}-R^+_{imjk}v^{\ \ m}_{:l}-R^+_{imjk:l}v^m)
\\
=& v_{:}^{\ ij}(v_{:kilj}-R^+_{kmjl}v^{\ \ m}_{:i}-R^+_{imjl}v^{\ \ m}_{:k}-R^+_{imjk}v^{\ \ m}_{:l}-R^+_{imjk:l}v^m)
\\
=& v_{:}^{\ ij}[(v_{:kli}-R^+_{kmil}v^m)_{:j} -R^+_{kmjl}v^{\ \ m}_{:i}-R^+_{imjl}v^{\ \ m}_{:k}-R^+_{imjk}v^{\ \ m}_{:l}-R^+_{imjk:l}v^m]
\\
=& v_{:}^{\ ij}(v_{:klij}-R^+_{kmil}v^{\ \ m}_{:j}- R^+_{kmil:j}v^m -R^+_{kmjl}v^{\ \ m}_{:i}-R^+_{imjl}v^{\ \ m}_{:k}-R^+_{imjk}v^{\ \ m}_{:l}-R^+_{imjk:l}v^m).
\end{aligned}
$$
So 
$$
\begin{aligned}
v_{:}^{\ ij}v_{:ijkl}=&v{:}^{\ ij}(v^{\ k}_{:\ kij}+R^+_{mi}v^{\ \ m}_{:j}+R^+_{mi:j}v^m+R^+_{mj}v^{\ \ m}_{:i}-2R^+_{imjk}v_{:}^{\ mk}-R_{imjk:}^{+\ \ \ \ \ k}v^m)
\\
=&v_{:}^{\ ij}(v^{\ k}_{:\ kij}-2nv_{:ij}+2v_{:k}^{\ \ k}[g_+]_{ij}-2v_{:ij}-2W^+_{imjk}v{:}^{\ mk}). 
\end{aligned}
$$
Since $v_{:k}^{\ \ k}=2(n+2)v-w$, we have
$$
\begin{aligned}
v_{:}^{\ ij}v_{:ijk}^{\ \ \ \ k}
&=2|\nabla^2 v|^2_+   -\langle\nabla^2 v,\nabla^2w\rangle_+  +2[2(n+2)v-w]^2-2W^+_{imjk}v_{:}^{\ ij}v_{:}^{\ mk}. 
\end{aligned}
$$
Hence
$$
\begin{aligned}
(|\nabla^2 v|_+^2)_{:k}^{\ \ k} =2\big(|\nabla^3v|_+^2 + 2|\nabla^2 v|^2_+   -\langle\nabla^2 v,\nabla^2w\rangle_+  +2[2(n+2)v-w]^2-2W^+_{imjk}v{:}^{\ mk}\big). 
\end{aligned}
$$
Moreover,
$$
(\langle \nabla v, \nabla w\rangle_+)_{:k}^{\ \ k} = v_{:ik}^{\ \ \ k}w^i+2v_{:ik}w{:}^{\ ik}+v^iw_{:ik}^{\ \ \ k} =2\langle \nabla^2v, \nabla^2w\rangle_++4\langle\nabla v,\nabla w\rangle_+-|\nabla w|^2_+, 
$$
$$
(w^2)_{:k}^{\ \ k}=2w_kw^k+2ww_{:k}^{\ \ k}=2|\nabla w|^2_+. 
$$
And we get the formula for  $\triangle_+II=-II_{:k}^{\ \ k}$. 
\end{proof}

\begin{lemma}
Assume the $(X^{n+1}, g_+)$ is locally conformally flat; $J$ is a positive constant and $Q_4\geq 0$ on $M$.  Then the test function $\psi$ satisfies
$$
\begin{gathered}
\triangle_+\psi-s(n-s)\psi<0
\end{gathered}
$$
all over $X$. 
\end{lemma}
\begin{proof}
Since $J$ is a positive constant,  
 $w=w_0$ is also a positive constant. 
By Lemma \ref{lem.2},  $I|_{M}=K< 0$ and by Lemma \ref{lem.3}, 
$$
\begin{aligned}
II=\triangle_+I=&(n+2-s)\left(\frac{1}{2}|\nabla^2v|^2_+ +\frac{n}{2}|\nabla v|_+^2-4(n+2)v^2\right)
+[3(n+2)-2s]vw-\frac{1}{2}w^2. 
\end{aligned}
$$
Since $(X^{n+1}, g_+)$ is locally conformally flat,  the Weyl tensor $W^+=0$. By Lemma \ref{lem.4}, 
$$
\begin{aligned}
\triangle_+II+2(n+2)II& = (n+2-s)\Big( -|\nabla^3v|^2_+ +(6n+16)|\nabla v|^2 \Big) =-(n+2-s)|V|_+^2\leq 0, 
\end{aligned}
$$
where 
$$
V_{ijk}=v_{:ijk}-v_{i}[g_+]_{jk}-v_{j}[g_+]_{ik}-2v_{k}[g_+]_{ij}.
$$ 
Notice that $II|_{M}=0$. 
Therefore, by maximum principle 
$$
II=\triangle_+I \leq 0, 
$$
all over $X$, which together with $I|_{M}=K<0$ implies that 
$$
 I<0
$$
all over $X$. 

\end{proof}

\begin{proposition}\label{prop.2.2}
Let $(X^{n+1}, g_+)$ ($n\geq 5$) be a locally conformally flat Poinar\'{e}-Einstein manifold with conformal infinity $(M, [g])$. Fix a representative $g$ for the conformal infinity. Assume the scalar curvature $R$ is a positive constant and $Q_4\geq 0$ on $M$. Then for all  $\gamma\in (1,2)$, $Q_{2\gamma}>0$. 
\end{proposition}

\begin{proof}
Similar  as Guillarmou-Qing did in \cite{GQ1}, we
compare the two functions $u$ and $\psi$, which are defined in (\ref{eq.1}) and (\ref{eq.2}) with $\lambda=n+2$. First $u/\psi$ satisfies the equation:
$$
\triangle_+\left(\frac{u}{\psi}\right) = \left(s(n-s)-\frac{\triangle_+\psi}{\psi}\right)\frac{u}{\psi}
+2\nabla\left(\frac{u}{\psi}\right) \frac{\nabla\psi}{\psi}, 
\quad \left(\frac{u}{\psi}\right)_{M}=1. 
$$
Notice that $u/\psi>0$. Applying maximum principle, we show that $u/\psi$ can not attain an interior positive minimum. Hence $u/\psi\geq 1$, i.e. $u\geq \psi$. Near the boundary, this means
$$
1+x^2u_2+x^{2\gamma}u_{2\gamma}+x^4u_4+O(x^{5})
\geq 1+x^2\psi_2+x^4\psi_4+O(x^{5})
$$
Since $\psi_2=u_2$ and $\psi_4>u_4$, we have $u_{2\gamma}>0$. Hence  $Q_{2\gamma}>0$ on $M$. 
\end{proof}

\textbf{Proof of Theorem \ref{thm.1}}: Part (a) is proved in Proposition \ref{prop.2.2}. Part (b) (c) (d) are from Proposition \ref{prop.2.2} and Proposition \ref{prop.2.1}. Part (e) is proved by the same proof for Proposition 4.2 in \cite{GQ1}. Recall the positive results in \cite{GQ1}. Then in our setting, $P_{2\gamma}$ is positive for all $\gamma\in [0, 2)$ and hence the first scattering pole $s_0\leq \frac{n}{2}$. Furthermore, if $Q_4(p)>0$ at some point $p$, then $P_4$ is also positive. Therefore the first scattering pole $s_0< \frac{n}{2}$.


\end{document}